\documentclass[11 pt]{amsart}%
\usepackage{eurosym}
\usepackage{graphicx}
\usepackage{mathrsfs}
\usepackage{amssymb}
\usepackage[all]{xy}
\usepackage{amsmath}
\usepackage{amsfonts}%
\providecommand{\U}[1]{\protect\rule{.1in}{.1in}}
\newtheorem{thm}{Theorem}[section]
\newtheorem{coro}[thm]{Corollary}
\newtheorem{lem}[thm]{Lemma}

\newtheorem{prop}[thm]{Proposition}
\theoremstyle{definition}
\newtheorem{definition}[thm]{Definition}
\newtheorem{rem}[thm]{Remark}
\newtheorem{example}[thm]{Example}
\numberwithin{equation}{section}

\begin{document}
\title{ New properties on normalized null hypersurfaces.}
\author{Cyriaque Atindogbe}
\address{Cyriaque Atindogbe: Universit\'{e} d'Abomey-Calavi.}
\email{atincyr@imsp-uac.org}
\author{Manuel Guti\'{e}rrez}
\address{Manuel Guti\'{e}rrez: Universidad de M\'{a}laga.}
\email{m\_gutierrez@uma.es}
\author{Raymond Hounnonkpe}
\address{Raymond Hounnonkpe: Universit\'{e} d'Abomey-Calavi. Institute de
Math\'{e}matiques et de Sciences Physiques (IMSP).}
\email{rhounnonkpe@ymail.com}

\begin{abstract}
Rigging technique introduced in \cite{bi0} is a convenient way to address the
study of null hypersurfaces. It offers in addition the extra benefit of
inducing a Riemannian structure on the null hypersurface which is used to
study geometric and topological properties on it. In this paper we develop
this technique showing new properties and applications. We first discuss the
very existence of the rigging fields under prescribed geometric and
topological constraints. We consider the completeness of the induced rigged
Riemannian structure. This is potentially important because it allows to use
most of the usual Riemannian techniques.

\end{abstract}
\maketitle

\noindent{\small \textbf{{Keywords.~}}} Lorentzian manifolds, null
hypersurface, normalization, rigging vector field, rigged vector field, completeness.

\noindent{\small \textbf{{MSC(2010).~}}} 53C50, 53B30, 53B50.

\section{Introduction}

A null hypersurface in a spacetime is a smooth codimension one submanifold
such that the ambient metric degenerates when restricted to it. Null
hypersurfaces play an important role in general relativity, as they represent
horizons of various sorts (event horizon of a black hole, Killing horizon,
etc.) and include lightcones. The main drawback to study them as part of
standard submanifold theory is the degeneracy of the induced metric. Some
attempts to overcome this difficlty have had remarkable succes. In \cite{DB},
the approach consists in fixing a geometric data formed by a null section and
a screen distribution on the null hypersurface. This allows to induce some
geometric objects such as a connection, a null second fundamental form and
Gauss-Codazzi type equations. In \cite{kupeli} the author uses the quotient
vector bundle $TM/TM^{\perp}$ to "get rid" of the degeneracy of the induced
metric. Returning to the approach in \cite{DB}, the basic question is how to
reduce as much as possible the arbitrary choices and to have a reasonable
coupling between the properties of the null hypersurface and the ambient
space. In \cite{bi0}, the authors used the rigging technique to study null
hypersurfaces. It is based on the arbitrary choice of a unique vector field in
a neighborhood of the null hypersurface, called rigging vector field, from
which is constructed both a null section defined on the null hypersurface,
called rigged vector field, and a screen distribution. This rigging technique
has also the advantage to induce on the whole null hypersurface a Riemannian
structure coupled with the rigging, which is used as a bridge to study the
null hypersurface. The null geometry of the hypersurface is related to the
properties of the induced Riemannian structure on the hypersurface, allowing
handle it using Riemannian geometry. The question now arise of knowing wether
it is always possible to operate a choice of a rigging vector field with fixed
geometric properties (closedness, conformality, causality conditions, etc. )
but also with geometric prescribed properties for the induced rigged
Riemannian structure (completeness, pinching constraints, geodesibility,
etc.). This is our concern in the present paper. The fact that there is a
positive answer to a reasonable amount of the above questions reinforces our
opinion that the rigging technique can be a good tool in this theory.

In Section \ref{riggedstructure} we review some facts about null
hypersurfaces, fix notations and give two technical lemmas. Obstruction
results involving both topology and prescribed geometric conditions on the
rigging vector field are established in Section \ref{vanishingtheorems}, e.g.
Theorem \ref{thm1.1}. The completeness properties of the induced Riemannian
metric are considered in Section \ref{completeness}. The first part is
concerned with some splitting results on the hypersurface equipped with its
rigged Riemannian structure. This allows us to get completeness sufficient
conditions in Robertson-Walker spaces, e.g. Theorem \ref{warp1} and
Proposition \ref{warp2}. After this, we consider the case of
Generalized-Robertson-Walker (GRW) spaces. We show that there are natural
rigging using the warping function leading to a complete induced Riemannian
structure e.g. Proposition \ref{grw1}. Finally, we show using closedness
argument on the rigging field and compacity of the screen leaves that the
induced Riemannian structure is complete e.g. Theorem \ref{main}. In
Section~\ref{applic}, we establish some results on null hypersurfaces under
completeness assumption of the induced Riemannian metric. In
subsection~\ref{ricciestimates}, we establish some estimates on mean curvature
on null hypersurfaces with complete rigged Riemannian structure e.g.
Theorem~\ref{thm1}, Theorem~\ref{thm3}, Theorem~\ref{iso}. Finally,
Section~\ref{semidefinitesharp} deals with null hypersurfaces for which the
screen shape operator is semi definite. We prove some obstruction results on
the existence of closed geodesics e.g. Proposition~\ref{prop3.1}, and show
under a completeness condition that the manifold structure of the null
hypersurface (say) $M$ in a 3-dimensional simply connected Lorentzian manifold
with no closed null curve is diffeomorphic to the plane or the cylinder e.g.
Theorem~\ref{thm4b}. Finally, we investigate about the existence of
topologically closed totally geodesic null hypersufaces in Robertson-Walker
spaces and prove non existence of lightlike line in some cases e.g.
Theorem~\ref{rober} and Corollary~\ref{corline}.

\section{ Normalization and rigged Riemannian structure}

\label{riggedstructure} Let $(\overline{M},\overline{g})$ be a $(n+2)$%
-dimensional Lorentzian manifold and $M$ a null hypersurface in $\overline{M}%
$. This means that at each $p\in M$, the restriction $\overline{g_{p}}%
_{|T_{p}M}$ is degenerated, that is there exists a non-zero vector $U\in
T_{p}M$ such that $\overline{g}(U,X)=0$ for all $X\in T_{p}M$. Hence, in null
setting, the normal bundle $TM^{\perp}$ of the null hypersurface $M^{n+1}$ is
a rank $1$ vector subbundle of the tangent bundle $TM$, contrary to the
classical theory of non-degenerate hypersurfaces for which the normal bundle
has trivial intersection $\{0\}$ with the tangent one playing an important
role in introducing the main induced geometric objects on $M$. Let us start
with the usual tools involved in the study of such hypersurfaces according
to~\cite{DB}. They consist in fixing on the null hypersurface a geometric data
formed by a null section and a screen distribution. By \textit{screen
distribution} on $M^{n+1}$, we mean a complementary bundle of $TM^{\perp}$ in
$TM$. It is then a rank $n$ non-degenerate distribution over $M$. In fact,
there are infinitely many possibilities of choices for such a distribution.
Each of them is canonically isomorphic to the factor vector bundle
$TM/TM^{\perp}$. For reasons that will become obvious in few lines below, let
us denote such a distribution by $\mathscr{S}(N)$. We then have,
\begin{equation}
TM=\mathscr{S}(N)\oplus TM^{\perp},\label{eq:4}%
\end{equation}
where $\oplus$ denotes the orthogonal direct sum. From \cite{DB}, it is known
that for a null hypersurface equipped with a screen distribution, there exists
a unique rank~$1$ vector subbundle $tr(TM)$ of $T\overline{M}$ over $M$, such
that for any non-zero section $\xi$ of $TM^{\perp}$ on a coordinate
neighborhood ${\mathscr{U}}\subset M$, there exists a unique section $N$ of
$tr(TM)$ on ${\mathscr{U}}$ satisfying
\begin{equation}
\overline{g}(N,\xi)=1,\quad\overline{g}(N,N)=\overline{g}(N,W)=0,\quad
\quad\forall W\in\mathscr{S}(N)|_{\mathscr{U}}).\label{eq:3}%
\end{equation}

Then $T\overline{M}$ admits the splitting:
\begin{equation}
T\overline{M}|_{M}=TM\oplus tr(TM)=\{TM^{\perp}\oplus tr(TM)\}\oplus
\mathscr{S}(N). \label{eq:5}%
\end{equation}

We call $tr(TM)$ a \textit{(null) transverse vector bundle} along $M$. In
fact, from (\ref{eq:3}) and (\ref{eq:5}) one shows that, conversely, a choice
of a transversal bundle $tr(TM)$ determines uniquely the screen distribution
$\mathscr{S}(N)$. A vector field $N$ as in (\ref{eq:3}) is called a
\textit{null transverse vector field} of $M$. It is then noteworthy that the
choice of a null transversal vector field $N$ along $M$ determines both the
null transverse vector bundle, the screen distribution $\mathscr{S}(N)$ and a
unique radical vector field, say $\xi$, satisfying (\ref{eq:3}).

Before continuing our discussion, we need to clarify the (general) concept of
rigging for our null hypersurface.

\begin{definition}
\label{rigging1} Let $M$ be a null hypersurface in a Lorentzian manifold. A
rigging for $M$ is a vector field $\zeta$ defined on some open set containing
$M$ such that $\zeta_{p}\notin T_{p}M$ for each $p\in M$.
\end{definition}

Given a rigging $\zeta$ in a neighborhood of $M$ in $(\overline{M}%
,\overline{g})$, let $\alpha$ denote the $1$-form $\overline{g}$-metrically
equivalent to $\zeta$, i.e $\alpha=\overline{g}(\zeta,~.)$. Take
$\omega=i^{\star}\alpha$, being $i:M\hookrightarrow\overline{M}$ the canonical
inclusion. Next, consider the tensors
\begin{equation}
\overset{\smile}{g}=\overline{g}+\alpha\otimes\alpha\quad\mbox{and}\quad
\widetilde{g}=i^{\star}\overset{\smile}{g}. \label{garc}%
\end{equation}

It is easy to show that $\widetilde{g}$ defines a Riemannian metric on the
(whole) hypersurface $M$. The \textit{rigged vector field} of $\zeta$ is the
$\widetilde{g}$-metrically equivalent vector field to the $1$-form $\omega$
and it is denoted by $\xi$. In fact the rigged vector field $\xi$ is the
unique lightlike vector field in $M$ such that $\overline{g}(\zeta,\xi)=1$.
Moreover, $\xi$ is $\widetilde{g}$-unitary. A screen distribution on $M$ is
given by
\[
\mathscr{S}(\zeta)=TM\cap\zeta^{\perp}.
\]
It is the $\widetilde{g}$-orthogonal subspace to $\xi$ and the corresponding
null transverse vector field to $\mathscr{S}(\zeta)$ is
\begin{equation}
N=\zeta-\frac{1}{2}\overline{g}(\zeta,\zeta)\xi.\label{nulrigging1}%
\end{equation}

A null hypersurface $M$ equipped with a rigging $\zeta$ is said to be
normalized and is denoted $(M,\zeta)$ (the latter is called a normalization of
the null hypersurface). A normalization $(M,\zeta)$ is said to be closed
(resp. conformal) if the rigging $\zeta$ is closed i.e the $1$-form $\alpha$
is closed (resp. $\zeta$ is a conformal vector field, i.e there exists a
function $\rho$ on the domain of $\zeta$ such that $L_{\zeta}\overline
{g}=2\rho\overline{{g}}$). We say that $\zeta$ is a \textit{null rigging} for
$M$ if the restriction of $\zeta$ to the null hypersurface $M$ is a null
vector at each point in $M$.

Let $\zeta$ be a rigging for a null hypersurface in a Lorentzian manifold
$(\overline{M},\overline{g})$. The screen distribution $\mathscr{S}(\zeta
)=\ker\omega$ is integrable whenever $\omega$ is closed, in particular if the
rigging is closed. On a normalized null hypersurface $(M,\zeta)$, the Gauss
and Weingarten formulas are given by
\begin{align}
\overline{\nabla}_{X}Y  &  =\nabla_{X}Y+B(X,Y)N,\label{eq:6a}\\
\overline{\nabla}_{X}N  &  =-A_{N}X+\tau(X)N,\label{eq:6b}\\
\nabla_{X}PY  &  =\overset{\star}{\nabla}_{X}PY+C(X,PY)\xi,\label{eq:6c}\\
\nabla_{X}\xi &  =-\overset{\star}{A}_{\xi}X-\tau(X)\xi, \label{eq:6d}%
\end{align}
for any $X,Y\in\Gamma(TM)$, where $\overline{\nabla}$ denotes the Levi-Civita
connection on $(\overline{M},\overline{g})$, $\nabla$ denotes the connection
on $M$ induced from $\overline{\nabla}$ through the projection along the null
transverse vector field $N$ and $\overset{\star}{\nabla}$ denotes the
connection on the screen distribution $\mathscr{S}(\zeta)$ induced from
$\nabla$ through the projection morphism $P$ of $\Gamma(TM)$ onto
$\Gamma\left(  \mathscr{S}(\zeta)\right)  $ with respect to the decomposition
(\ref{eq:4}). The $(0,2)$ tensor $B$ is the null second fundamental form on
$TM$, $\overset{\star}{A}_{\xi}$ the shape operator on $TM$ with respect to
the rigged vector field $\xi$ and $\tau$ a $1$-form on $TM$ defined by
\[
\tau(X)=\overline{g}(\overline{\nabla}_{X}N,\xi).
\]

The null second fundamental form $B$ is symmetric, whereas the tensor $C$ is
not in general. The following holds
\begin{equation}
B(X,Y)=g(\overset{\star}{A}_{\xi}X,Y),\quad C(X,PY)=g(A_{N}X,Y)\ \ \forall
X,Y\in\Gamma(TM), \label{eq:6fo}%
\end{equation}
and
\begin{equation}
B(X,\xi)=0,\quad\overset{\star}{A}_{\xi}\xi=0. \label{eqnul}%
\end{equation}

It follows from (\ref{eqnul}) that the integral curves of $\xi$ are
pregeodesics in both $\overline{M}$ and $(M,\nabla)$, as $\overline{\nabla
}_{\xi}\xi=\nabla_{\xi}\xi=-\tau(\xi)\xi$.

A null hypersurface $M$ is said to be \textit{totally umbilic} (resp.
\textit{totally geodesic}) if there exists a smooth function $\rho$ on $M$
such that at each $p\in M$ and for all $u,v\in T_{p}M$, $B(p)(u,v)=\rho
(p)\overline{g}(u,v)$ (resp. $B$ vanishes identically on $M$). These are
intrinsic notions on any null hypersurface in the sense that they are
independent of the normalization. Note that $M$ is \textit{totally umbilic}
(resp. \textit{totally geodesic}) if and only if $\overset{\star}{A}_{\xi
}=\rho P$ (resp. $\overset{\star}{A}_{\xi}=0$). It is noteworthy to mention
that the shape operators $\overset{\star}{A}_{\xi}$ and $A_{N}$ are
$\mathscr{S}(\zeta)$-valued.

The induced connection $\nabla$ is torsion-free, but it does not preserve
$\overline{g}$ except $M$ is totally geodesic. In fact we have for all tangent
vector fields $X,Y$ and $Z$ in $TM$,
\begin{equation}
(\nabla_{X}\overline{g})(Y,Z)=B(X,Y)\omega(Z)+B(X,Z)\omega(Y). \label{eq:derc}%
\end{equation}

The trace of $\overset{\star}{A}_{\xi}$ is the null mean curvature of $M$,
explicitly given by
\[
H_{p}=\sum_{i=2}^{n+1}\overline{g}(\overset{\star}{A}_{\xi}(e_{i}),e_{i}%
)=\sum_{i=2}^{n+1}B(e_{i},e_{i}),
\]
being $(e_{2},\dots,e_{n+1})$ an orthonormal basis of $\mathscr{S}(\zeta)$ at
$p$.

The (shape) operator $\overset{\star}{A}_{\xi}$ is self-adjoint as the null
second fundamental form $B$ is symmetric. However, this is not the case for
the operator $A_{N}$ as it is shown in the following lemma.

\begin{lem}
\cite{bi11} \label{permutan} For all $X,Y\in\Gamma(TM)$,
\begin{equation}
\overline{g}(A_{N}X,Y)-\overline{g}(A_{N}Y,X)=\tau(X)\alpha(Y)-\tau
(Y)\alpha(X)-d\alpha(X,Y). \label{eq:permutan}%
\end{equation}

\end{lem}

In case the normalization is closed the $1$-form $\tau$ is related to the
shape operator of $M$ as follows.

\begin{lem}
\label{tau1} Let $(M,\zeta)$ be a closed normalization of a null hypersurface
$M$ in a Lorentzian manifold. Then
\begin{equation}
\tau=-\overline{g}(A_{N}\xi,~\cdot~)+\tau(\xi)\alpha. \label{tau2}%
\end{equation}

In particular if $\tau(\xi)=0$, then $A_{N}\xi=-\tau^{\sharp_{\widetilde{g}}}$.
\end{lem}

\begin{proof}
By the closedness of $\alpha$, the condition
\[
X\cdot\alpha(Y)-Y\cdot\alpha(X)-\alpha([X,Y])=0
\]
is equivalent to%
\[
\overline{g}(\overline{\nabla}_{X}N,Y)=\overline{g}(\overline{\nabla}%
_{Y}N,X).
\]

Then by the weingarten formula, we get
\[
\overline{g}(-A_{N}X,Y)+\tau(X)\alpha(Y)=\overline{g}(-A_{N}Y,X)+\tau
(Y)\alpha(X).
\]

In this relation, take $Y=\xi$ to get
\[
\tau(X)=-\overline{g}(A_{N}\xi,X)+\tau(\xi)\alpha(X)
\]
which gives the desired formula.

Assume now that $\tau(\xi)=0$. Then,
\[
\tau(X)=-\overline{g}(A_{N}\xi,X)=-\widetilde{g}(A_{N}\xi,X),
\]
for all $X\in TM$, as $A_{N}\xi\in\mathscr{S}(\zeta)$. Hence, $A_{N}\xi
=-\tau^{\sharp_{\widetilde{g}}}$.
\end{proof}

\section{Compact null hypersufaces}

\label{vanishingtheorems}The existence of a rigging vector for a null
hypersurface is the first step in the rigging technique. It is clear that in
general it is not possible to choose a rigging vector field, so it is
interesting to identify the situations where you can not choose it. Despite
the trivial cases where there is an obstruction due to the existence of a
nowhere zero vector, the rigged vector, on a compact null hypersurface which
force it to have zero Euler characteristic, there are more subtle situations
that we explore here. Given a compact null hypersurface $M$ in a Lorentzian
manifold $(\overline{M},\overline{g})$, we give some restriction on the
geometric properties of the admissible rigging of $M$ due to the topology of
$M$ that can prevent the existence of some kind of normalization.

\label{closednormalization}

\begin{thm}
\label{thm1.1} Let $(\overline{M},\overline{g})$ be a Lorentzian manifold and
$M$ a compact null hypersurface in $\overline{M}$ with trivial first De Rham
cohomology group $H^{1}(M,\mathbb{R})$. Then $M$ admits no closed normalization.
\end{thm}

\begin{proof}
Suppose $M$ is compact and suppose that $b_{1}(M)=0$. If there exists a closed
rigging $\zeta$ then the $1$-form $\omega=i^{\star}\alpha$ is closed and there
exists a function $f$ on $M$ such that $df=\omega$, that is $\widetilde
{\nabla}f=\xi$. As a consequence, we have $\widetilde{g}(\widetilde{\nabla
}f,\widetilde{\nabla}f)=1$ which is not possible as $f$ has at least one
critical point on the compact manifold $M$.
\end{proof}

This allows us to prove the following result.

\begin{coro}
\label{cor1} Let $(\overline{M},\overline{g})$ be a Lorentzian manifold with a
closed timelike vector field. Then there is no compact simply connected null
hypersurface in $\overline{M}$.
\end{coro}

\begin{proof}
Let $M$ be a compact null hypersurface in $\overline{M}$. Since $(\overline
{M},\overline{g})$ has a closed timelike vector field say $\zeta$, the later
can be used (due to signature considerations) as a rigging for $M$. If we
suppose in addtion $\pi_{1}(M)=0$ which implies that $b_{1}(M)=0$ we get a
contradiction using above Theorem~\ref{thm1.1}. We conclude that if $M$ is
compact then $b_{1}(M)\geq1$ and hence $M$ is not simply connected.
\end{proof}

\begin{rem}
In fact, the above proof shows that if $(\overline{M},\overline{g})$ has a
closed timelike vector field, then there is no compact null hypersurface with
trivial first De Rham cohomology group in $\overline{M}$.
\end{rem}

\begin{prop}
\label{prop2} Let $(\overline{M},\overline{g})$ be a simply connected
Lorentzian manifold, then there is no closed normalization for any compact
null hypersurface in $\overline{M}$.
\end{prop}

\begin{proof}
Suppose there exists a compact null hypersurface in $\overline{M}$ with a
closed normalization $\zeta$, then $\alpha=\overline{g}(\zeta,.)$ is a closed
$1$-form on $\overline{M}$ and since $\overline{M}$ is simply connected, there
exists $f:\overline{M}\longrightarrow\mathbb{R}$ such that $\alpha=df$. This
imply that the $\widetilde{g}$-equivalent $1$-form $\omega$ to the rigged
vector field $\xi$ satisfies
\[
\omega=i^{\star}\alpha=i^{\star}df=d(i^{\star}f)=d(f\circ i).
\]

It follows that $\widetilde{\nabla}(f\circ i)=\xi$ and then $\widetilde
{g}(\widetilde{\nabla}(f\circ i),\widetilde{\nabla}(f\circ i))=1$ which is a
contradiction because $f\circ i$ has at least one critical point.
\end{proof}

\begin{definition}
[\cite{ad1}]A normalized null hypersurface $(M,\zeta)$ of a semi-Riemannian
manifold, is screen conformal if the shape operators $A_{N}$ and
$\overset{\star}{A}_{\xi}$ are related by
\begin{equation}
A_{N}=\varphi\overset{\star}{A}_{\xi} \label{eq:22n}%
\end{equation}
where $\varphi$ is a non-vanishing smooth function on $M$.
\end{definition}

As stated in the following theorem, such class of lightlike hypersurfaces has
a geometry which is essentially the same as that of their chosen screen distribution.

\begin{thm}
[\cite{ad1}]Let $(M,S(\zeta))$ be a screen conformal lightlike hypersurface of
a semi-Riemannian manifold $(\overline{M},\overline{g})$. Then the screen
distribution is integrable. Moreover, $M$ is totally geodesic (resp. totally
umbilical or minimal) in $\overline{M}$ if and only if any leaf $M^{\prime}$
of $S(\zeta)$ is totally geodesic (resp. totally umbilical or minimal) in
$\overline{M}$ as a codimension 2 nondegenerate submanifold.
\end{thm}

Our aim is to show that in a 4-dimensional Lorentzian manifold, compact null
hypersurfaces with finite fundamental group can not admit such normalizations.
We prove first the following:

\begin{prop}
\label{prop3} Let $\zeta$ be a rigging for a compact null hypersurface $M$ in
a Lorentzian manifold $(\overline{M},\overline{g})$ of constant curvature. If
the screen $\mathscr{S}(\zeta)$ is conformal and the first De Rham cohomology
group $H^{1}(M,\mathbb{R})$ is trivial, then $M$ is totally geodesic.
\end{prop}

\begin{proof}
Since $M$ is screen conformal there exists a non vanishing function $\rho$
defined on $M$ such that $A_{N}=\rho\overset{\star}{A}_{\xi}$. From the
Gauss-Codazzi equations, see \cite[Page 95, Eq. (3.12)]{DB}
\begin{align}
\overline{g}(\overline{R}(X,Y)\xi,N) &  =C(Y,\overset{\star}{A}_{\xi
}X)-C(X,\overset{\star}{A}_{\xi}Y)\nonumber\label{traf1}\\
&  -2d\tau(X,Y),\quad\forall X,Y\in\Gamma(TM).
\end{align}

But the left hand side of (\ref{traf1}) vanishes since $\overline{M}$ has
constant curvature, and $\xi$ is orthogonal to both $X$ and $Y$. Moreover,
\[
C(Y,\overset{\star}{A}_{\xi}X)-C(X,\overset{\star}{A}_{\xi}Y)=\overline
{g}(A_{N}Y,\overset{\star}{A}_{\xi}X)-\overline{g}(A_{N}X,\overset{\star}%
{A}_{\xi}Y)=0
\]
since $A_{N}=\rho\overset{\star}{A}_{\xi}$. Using the fact that $H^{1}%
(M,\mathbb{R})$ is trivial, there exists a function (say) $\phi$ defined on
$M$ such that $\tau=d\phi$. Define a new rigging vector field by
$\widehat{\zeta}=\exp(-\phi)\zeta$, so $\widehat{N}=\exp(-\phi)N$. Moreover,
it follows (from \cite{At}, Lemma~2.1) that $\widehat{\tau}=\tau+d(\ln
(\exp(-\phi))=0$ as $\tau=d\phi$. Denote respectively by $\widehat{\xi
},\widehat{H},\overset{\star}{A}_{\widehat{\xi}}$ the rigged vector field, the
mean curvature function and the screen shape operator form of $\widehat{\zeta
}$, we have (\cite{bi11}, Remark~3)
\[
\overline{Ric}(\widehat{\xi})=\widehat{\xi}(\widehat{H})+\widehat{\tau
}(\widehat{\xi})\widehat{H}-\lvert\overset{\star}{A}_{\widehat{\xi}}\lvert
^{2}.
\]

But $\overline{Ric}(\widehat{\xi})=0$ since $\overline{M}$ has constant
curvature and $\widehat{\tau}(\widehat{\xi})=0,$ it follows that $\left.
\widehat{\xi}(\widehat{H})-\lvert\overset{\star}{A}_{\widehat{\xi}}\lvert
^{2}=0.\right.  $ Using the inequality $\left.  \lvert\overset{\star}%
{A}_{\widehat{\xi}}\lvert^{2}\geq\frac{1}{n}\widehat{H}^{2}\right.  $, we
obtain $\left.  \widehat{\xi}(\widehat{H})-\frac{1}{n}\widehat{H}^{2}%
\geq0\right.  $, and since $\widehat{\xi}$ is complete ($M$ being compact) we
get that $\widehat{H}=0$. From the relation $\left.  \widehat{\xi}(\widehat
{H})-\lvert\overset{\star}{A}_{\widehat{\xi}}\lvert^{2}=0,\right.  $ it
follows that $\lvert\overset{\star}{A}_{\widehat{\xi}}\lvert^{2}=0$ which
leads to $\overset{\star}{A}_{\widehat{\xi}}=0$. We conclude that $M$ is
totally geodesic.
\end{proof}

We can get now the following result.

\begin{prop}
\label{prop4} Let $(\overline{M}^{4},\overline{g})$ be a 4-dimensional
Lorentzian manifold of constant curvature and $M$ a compact null hypersurface.
If $M$ has finite fundamental group then there is no normalization such that
$M$ is screen conformal.
\end{prop}

\begin{proof}
Let $M$ be as above. Suppose there is a normalization such that $M$ is screen
conformal. Since $M$ has finite fundamental group, the first De Rham
cohomology group $H^{1}(M,\mathbb{R})$ is trivial. It follows from
Proposition~\ref{prop3} that $M$ is totally geodesic. Elsewhere, $M$ being
screen conformal, $\mathscr{S}(\zeta)$ is integrable and induces a foliation
on $M$. We show that the leaves of the screen distribution $\mathscr{S}(\zeta
)$ are totally geodesic in $(M,\widetilde{g})$. For this, recall (from
\cite{bi0}, Proposition~3.7)that for $X$ and $Y$ in $\mathscr{S}(\zeta)$,
\[
\widetilde{\nabla}_{X}Y=\overset{\star}{\nabla}_{X}Y-\widetilde{g}%
(\widetilde{\nabla}_{X}\xi,Y)\xi,
\]
but we also have
\[
\widetilde{g}(\widetilde{\nabla}_{X}\xi,Y)+\widetilde{g}(\widetilde{\nabla
}_{Y}\xi,X)=L_{\xi}\widetilde{g}(X,Y)=-2B(X,Y).
\]

Now, since $\mathscr{S}(\zeta)$ is integrable, we have $\left.  \widetilde
{g}(\widetilde{\nabla}_{X}\xi,Y)=\widetilde{g}(\widetilde{\nabla}_{Y}%
\xi,X)\right.  $. It follows that $\left.  \widetilde{g}(\widetilde{\nabla
}_{X}\xi,Y)=-B(X,Y)\right.  $ which implies that
\[
\widetilde{\nabla}_{X}Y=\overset{\star}{\nabla}_{X}Y+B(X,Y)\xi.
\]

In other words, the second fundamental form of each leaf of $\mathscr{S}(\zeta
)$ in $(M,\widetilde{g})$ is $B$ and then each of them is totally geodesic in
$(M,\widetilde{g})$ as $M$ is totally geodesic in $(\overline{M}^{4}%
,\overline{g})$. It follows that there exits a totally geodesic codimension
one foliation on the compact 3-manifold $M$, hence $M$ must have infinite
fundamental group (see \cite{DL}), which is a contradiction.
\end{proof}

\section{Completeness of $(M, \widetilde{g})$}

\label{completeness} On a normalized null hypersurface in a Lorentzian
manifold there is a bridge between the Riemannian geometry of the couple
$(M,\widetilde{g})$ and the null geometry of $M$. The key is to use Riemannian
techniques, so it worths to investigate on its completeness. We consider first
this problem in some particular Lorentzian manifold (Robertson-Walker spaces,
generalized Robertson-Walker spaces) and finish with the case of arbitrary
Lorentzian manifold.

\label{warpeddecomp} It is known that a totally umbilic null hypersurface with
a closed normalization splits locally as a twisted product, the decomposition
being global if $M$ is simply connected and the rigged vector field complete,
\cite[Theorem 5.3]{bi0}. We show here that if moreover it admits a closed
conformal rigging in an ambient space form, the local twisted product
structure of the rigged metric is in fact a warped product. Elsewhere, we show
that in a Robertson-Walker space case, using a specific rigging, we also get
warped decomposition of totally umbilic null hypersurfaces. This allows us to
state some sufficient conditions for $(M,\widetilde{g})$ to be complete.

\begin{thm}
\label{warp1} Let $(\overline{M}^{n+2},\overline{g})$ be a Lorentzian manifold
with constant curvature (with $n\geq2$) and $M$ a totally umbilic null
hypersurface admitting a closed conformal normalization $\zeta$. Then given
$p\in M$, the Riemannian structure$(M,\widetilde{g})$ is locally isometric to
a warped product $(\mathbb{R}\times S,dr^{2}+f^{2}g_{0})$ where $S$ is the
leaf of $\mathscr{S}(\zeta)$ through $p$, and $g_{0}$ is a conformal metric to
$g_{\lvert_{S}}$. Moreover, if $M$ is simply connected and the rigged vector
field $\xi$ complete, the decomposition is global.
\end{thm}

\begin{proof}
Using \cite[Theorem 5.3]{bi0}, the only point we are going to show is the
warped decomposition of $(M, \widetilde{g})$.

In \cite[Theorem 4.8]{bi0} it is shown that for $U,V\in TM$ the following
holds,
\[
R(U,V)\xi-\widetilde{R}(U,V)\xi=\overline{g}(\overline{R}(U,V)\xi,N)\xi
-\tau(U)\overset{\star}{A}_{\xi}(V)+\tau(V)\overset{\star}{A}_{\xi}(U).
\]

We also know that if $\zeta$ is closed and conformal, the $1-$form $\tau$
vanishes identically. Using the Gauss-Codazzi equation we have $\left.
R(U,V)\xi=\overline{R}(U,V)\xi.\right.  $ Finally $\overline{R}(U,V)\xi=0$
since $(\overline{M},\overline{g})$ has constant curvature and the above
equality becomes $\widetilde{R}(U,V)\xi=0$ for all tangent vector fields $U$
and $V$. Then, $\widetilde{Ric}(X,\xi)=0$ for all $\mathscr{S}(\zeta)-$valued
vector field $X$ (in fact for all tangent vector field $X$). The result
follows from the mixed Ricci flat condition, \cite[Theorem 1]{bi4}.
\end{proof}

\begin{rem}
\label{remgrad} The global decomposition of $(M, \widetilde{g})$ as warped
product still holds in Theorem~\ref{warp1} if $M$ is not simply connected but
the rigging is a gradient vector field (see \cite[Remark 5.4]{bi0}).
\end{rem}

In case $\overline{M}=I\times_{f}L$ is a Robertson-Walker space, we use the
classical rigging $\zeta=f\frac{\partial}{\partial t}$ wich is a gradient
conformal vector field to get the following.

\begin{prop}
\label{warp2} Let $\overline{M}=I\times_{f}L$ be a Robertson-Walker space and
$M$ a totally umbilic null hypersurface equipped with the (natural) rigging
$\zeta=f\frac{\partial}{\partial t}$.

\begin{enumerate}
\item Then $(M,\widetilde{g})$ is locally isometric to a warped product.
Moreover, if $\xi$ is complete, the decomposition is global.

\item If $\xi$ is complete and the screen distribution $\mathscr{S}(\zeta)$
(which is integrable) has compact leaves, then $(M,\widetilde{g})$ is complete.
\end{enumerate}
\end{prop}

\begin{proof}
1. Since $\overline{M}=I\times_{f}L$ is a Robertson-Walker space and $L$ being
of constant curvature $c$, we know that
\[
\overline{R}(U,V)W=\frac{(f^{\prime})^{2}+c}{f^{2}}(\overline{g}%
(V,W)U-\overline{g}(U,W)V)
\]
and $\overline{R}(U,V)\frac{\partial}{\partial t}=0$ for all $U,V,W$ tangent
to the factor $L$. Using the classical rigging $\zeta=f\frac{\partial
}{\partial t}$ wich is closed (in fact a gradient) and conformal, decompose
the associated rigged vector field as $\xi=a\frac{\partial}{\partial t}+X_{0}$
with $X_{0}\in TL$; we have $\overline{g}(X_{0},X)=0\;\forall\;X\in
\mathscr{S}(\zeta)$. Remark also that $\mathscr{S}(\zeta)\subset TL$. Taking
into account the above considerations, we get $\overline{R}(Y,X)\xi=0$,
$\forall X,Y\in\mathscr{S}(\zeta)$. Following the proof of previous
Theorem~\ref{warp1} we get that $\widetilde{R}(Y,X)\xi=0$, $\forall
X,Y\in\mathscr{S}(\zeta)$ and then $\widetilde{Ric}(X,\xi)=0$ for all
$\mathscr{S}(\zeta)$-valued vector field $X$. The conclusion follows as in the
last theorem. Taking into account Remark~\ref{remgrad} the global
decomposition holds if $\xi$ is complete.

2. From point 1, we have that $(M,\widetilde{g})$ is globally isometric to a
warped product $\left.  (\mathbb{R}\times S,dr^{2}+f^{2}g_{0})\right.  $, and
being $S$ compact, it is complete.
\end{proof}

\label{criterions}

We study now the $\widetilde{g}$-completeness of null hypersurfaces in
Generalized Robertson-Walker spaces. Let $\overline{M}=I\times_{f}L$ be a
Generalized-Robertson-Walker space and $M$ a null hypersurface of
$\overline{M}$. Take $h$ to be any primitive of $-f$. Then $\overline{\nabla
}h=f\frac{\partial}{\partial t}$. Using $\zeta=f\frac{\partial}{\partial t}$
as a rigging of $M$, we get $\widetilde{\nabla}(h\circ i)=\xi$ and
$\widetilde{g}(\widetilde{\nabla}(h\circ i),\widetilde{\nabla}(h\circ i))=1$
where $i$ is the canonical inclusion of $M$ in $\overline{M}$. Recall from
\cite{bi5,bi5 b} the following important fact: A Riemannian manifold $(M,g)$
is complete if and only if it supports a proper $C^{3}$ function say $f$ such
that $g(\nabla f,\nabla f)$ is bounded. Hence, if $h$ is proper on $M$ then
$(M,\widetilde{g})$ is complete. We have shown the following:

\begin{prop}
\label{grw1} Let $\overline{M}=I\times_{f}L$ be a Generalized-Robertson-Walker
space and $M$ a null hypersurface equipped with the rigging $\zeta
=f\frac{\partial}{\partial t}$ such that $h\circ i$ is a proper function on
$M$ (where $h$ is any primitive of $-f$). Then its rigged Riemannian structure
$(M,\widetilde{g})$ is complete.
\end{prop}

\begin{rem}
In case $L$ is compact, $h:\overline{M}\longrightarrow\mathbb{R}$ is proper if
and only if $\overline{M}$ is null complete. Recall also that if
$h:\overline{M}\longrightarrow\mathbb{R}$ is proper and $M$ is a closed subset
then $h\circ i:M\longrightarrow\mathbb{R}$ is also proper. Hence
\end{rem}

\begin{prop}
\label{grw2} Let $\overline{M}=I\times_{f}L$ be a Generalized-Robertson-Walker
space with compact Riemannian factor $L$. If $\overline{M}$ is null complete,
any topologically closed null hypersurface $M$ in $\overline{M}$ is
$\widetilde{g}$-complete for the usual rigging $\zeta=f\frac{\partial
}{\partial t}$.
\end{prop}

\begin{example}
Consider $\overline{M}=\mathbb{R}\times_{t^{2}+1}L$ with $L$ compact. It is
null complete and then any (topologically) closed null hypersurface $M$ in
$\overline{M}$ is $\widetilde{g}$-complete for the usual rigging $\zeta
=f\frac{\partial}{\partial t}$.
\end{example}

For GRW spaces with complete Riemannian factors, we show the following.

\begin{thm}
\label{grw3} Let $\overline{M}=\mathbb{R}\times_{f}L$ be a
Generalized-Robertson-Walker space with complete Riemannian factor $(L,g_{0})$
and $M$ be a topologically closed null hypersurface of $\overline{M}$. Then,
the Riemannian structure $(M,\widetilde{g})$ induced by the rigging
$\zeta=\sqrt{2}\dfrac{\partial}{\partial t}$ is complete.
\end{thm}

\begin{proof}
The Lorentzian metric on $\overline{M}$ is given by $\overline{g} = -dt^{2}+
f^{2}g_{0}$. Then using the rigging $\zeta= \sqrt{2}\dfrac{\partial}{\partial
t}$ and the first equality in (\ref{garc}) we get
\[
\overset{\smile}{g} = dt^{2}+ f^{2}g_{0}%
\]
which shows that $(\overline{M},\overset{\smile}{g})$ is a complete Riemannian
manifold as $(L,g_{0})$ is complete. Then, since $M$ is topologically closed,
using the second equality in (\ref{garc}) we see that $(M,\widetilde{g})$ is a
complete Riemannian manifold.
\end{proof}

The following theorem gives some sufficient conditions to get a complete
induced Riemannian structure on a given null hypersuface in any Lorentzian
manifold. It is an improvement of Proposition \ref{warp2} point 2.

\begin{thm}
\label{main} Let $(\overline{M}^{n+2},\overline{g})$ be a Lorentzian manifold
and $(M,\zeta)$ a closed normalization of a connected non compact null
hypersurface. If $\xi$ is complete and $\mathscr{S}(\zeta)$ has compact leaves
then $(M,\widetilde{g})$ is complete.
\end{thm}

\begin{proof}
Let $\Phi$ be the flow of $\xi$. Since $\xi$ is complete, closed with compact
orthogonal leaves,
\begin{align}
\Phi &  :\mathbb{R}\times L\longrightarrow M\nonumber\\
&  (t,p)\longmapsto\Phi_{t}(p)\nonumber
\end{align}
is a diffeomorphism ,where $L$ is a leaf of $\mathscr{S}(\zeta)$ \cite[Proof
of Lemma 3.1]{GutOlea2003a}, \cite[Theorem 4.1]{GutOlea2003b}. Suppose the
inverse of $\Phi$ decomposes as
\begin{align}
\Phi^{-1} &  :M\longrightarrow\mathbb{R}\times L\nonumber\\
&  x\longmapsto(f(x),\psi(x))\nonumber
\end{align}
so we have $\widetilde{\nabla}f=\xi$ which is nowhere zero, $f$ is a
submersion. Since $\left\vert \widetilde{\nabla}f\right\vert =1$ and $\xi$ is
complete, there exists a diffeomorphism $\left.  F:\mathbb{R}\times
f^{-1}(0)\rightarrow M\right.  $, \cite[Theorem 6.2]{bi12}. Moreover, $\left.
pr_{1}:\mathbb{R}\times f^{-1}(0)\rightarrow\mathbb{R}\right.  $ being the
projection on the first factor we have $f=pr_{1}\circ F^{-1}$. By hypothesis
$M$ is connected and $F$ a diffeomorphism so $f^{-1}(0)$ is connected too. It
follows that $f^{-1}(0)$ is a leaf of $\mathscr{S}(\zeta)$ which is compact,
so $pr_{1}$ is a proper map. Thus, $f$ is a proper map. Since its gradient is
bounded, we conclude that $(M,\widetilde{g})$ is complete.
\end{proof}

\begin{thm}
\label{thmmghc} Let $(\overline{M},\overline{g})$ be a Lorentzian manifold
furnished with a proper function $f$ whose gradient is timelike everywhere.
Then for any topologically closed null hypersurface in $\overline{M}$, the
rigging $\zeta=\overline{\nabla}f$ makes $(M,\widetilde{g})$ complete.
\end{thm}

\begin{proof}
Let us denote by $h$ the restriction of $f$ on $M$. Since $M$ is closed in
$\overline{M}$, $h$ is also a proper function on $M$. Considering the rigging
$\zeta=\overline{\nabla}f$, a straightforward argument shows that
$\widetilde{\nabla}h=\xi$, so $\widetilde{g}(\widetilde{\nabla}h,\widetilde
{\nabla}h)=1$. It follows that $h$ is a proper function on $M$ whose gradient
is bounded, then $(M,\widetilde{g})$ is complete.
\end{proof}

\section{Applications}

\label{applic} Let $(M,\zeta)$ be a normalized null hypersurface of a
Lorentzian manifold $\overline{M}$, we show several results under the
hypothesis that $(M,\widetilde{g})$ is a complete Riemannian manifold. In the
first part we show that the non-normalized null mean curvature of $M$ is
strongly controled by the ricci curvature of $\overline{M}$ evaluated on the
associated rigged vector field $\xi$. We investigate also about mean curvature
of null hypersurfaces all of whose screen principal curvatures are constant.
The second part deals with null hypersurface with semi-definite shape. Non
existence of closed geodesic in $(M,\widetilde{g})$ is proved for some special
cases which allows us to give a classification theorem e.g.
Theorem~\ref{thm4b} and corollary~\ref{classi}.Finally, we investigate about
the existence of topologically closed totally geodesic null hypersufaces in
Robertson-Walker spaces.

\subsection{Ricci estimates and mean curvature boundedness}

\label{ricciestimates} Given a normalized null hypersurface $(M,\zeta)$ of a
Lorentzian manifold $\overline{M}$, we prove some results about $M$ under
hypothesis on the ricci curvature of $\overline{M}$ evaluated on the
associated rigged vector field $\xi$.

\begin{thm}
\label{thm1} Let $(\overline{M},\overline{g})$ be a Lorentzian manifold and
$(M,\zeta)$ a closed normalization of a null hypersurface such that $\tau
(\xi)=0$. Assume $M$ to be $\widetilde{g}-$complete and there exists a
nonnegative constant $k$ such that $\left.  \overline{Ric}(\xi)\geq
-k\text{.}\right.  $ Then we have $\lvert H\lvert\leq k$ where $H$ stands for
the (non normalized) mean curvature of $M$.
\end{thm}

The proof uses the following.

\begin{thm}
[\cite{bi2}]\label{tak 1} Let $(M,g)$ be a complete connected Riemannian
manifold such that there exists $f:M\longrightarrow\mathbb{R}$ satisfying
$\left\vert \nabla f\right\vert =1$. Suppose that $\left.  Ric(\nabla f,\nabla
f)\geq-k\right.  $\thinspace\ ($k$ a nonnegative constant), then $\lvert\Delta
f\lvert\leq k$.
\end{thm}

We give now the proof of Theorem~\ref{thm1}.

\begin{proof}
Let\textbf{{ }}$\overline{\pi}:(\overline{M}^{\prime},\overline{g}^{\prime
})\rightarrow(\overline{M},\overline{g})$ be the semi-Riemannian universal
covering of $\overline{M}$. Define $M^{\prime}=\overline{\pi}^{-1}(M)$ which
is a null hypersurface because $\overline{\pi}$ is a local isometry and call
$i^{\prime}:M^{\prime}\rightarrow\overline{M}^{\prime}$ the canonical
inclusion. The closed rigging $\zeta$ can be lifted to a closed rigging
$\zeta^{\prime}$ on $M^{\prime}$. Call $\alpha^{\prime}=\overline{\pi}^{\ast
}\alpha$ its equivalent $1$-form being $\alpha$ the equivalent $1$-form to
$\zeta$. The rigged metric on $M^{\prime}$ is $\widetilde{g}^{\prime
}=i^{\prime\ast}(\overline{g}^{\prime}+\alpha^{\prime}\otimes\alpha^{\prime}%
)$. Using the following conmutative diagram
\[%
\xy\xymatrix{
{\overline M'} \ar[r]^{\overline{\pi}} & \overline M \\
M' \ar[u]^{i'} \ar[r]^{\pi} & M \ar[u]_i
}
\endxy
\]
where $\pi$ is the canonical projection from $M^{\prime}$ onto $M$, it is
clear that $\omega^{\prime}=\pi^{\ast}\omega$ where $\omega$ is the equivalent
$1$-form to the rigged field $\xi$ on $M$, and $\widetilde{g}^{\prime}%
=\pi^{\ast}\widetilde{g}$. If we call $\xi^{\prime}$ the rigged vector field
on $M^{\prime}$ induced by $\zeta^{\prime}$, we have $\pi_{\ast}\xi^{\prime
}=\xi\circ\pi$.

Using that $\pi$ is a local isometry, we have $\tau^{\prime}(\xi^{\prime})=0$,
$\widetilde{g}^{\prime}$ is complete and $\left.  \overline{Ric}^{\prime}%
(\xi^{\prime})\geq-k.\right.  $

Since $\overline{M}^{\prime}$ is simply connected and $\zeta^{\prime}$ is
closed, we know there exists $\left.  f:M^{\prime}\longrightarrow
\mathbb{R}\right.  $ such that $\widetilde{\nabla}^{\prime}f=\xi^{\prime}$ and
thus $\widetilde{g}^{\prime}(\widetilde{\nabla}^{\prime}f,\widetilde{\nabla
}^{\prime}f)=1$. Moreover, for a closed normalization, we have (see
\cite{bi0})%
\[
\ \overline{Ric}^{\prime}(\xi^{\prime})=\widetilde{Ric}^{\prime}(\xi^{\prime
})+\tau^{\prime}(\xi^{\prime})H^{\prime}%
\]
being $H^{\prime}$ the null mean curvature of $M^{\prime}$. Since
$\tau^{\prime}(\xi^{\prime})=0$, then we get $\left.  \overline{Ric}^{\prime
}(\xi^{\prime})=\widetilde{Ric}^{\prime}(\xi^{\prime})\right.  $. From this,
we get $\left.  \widetilde{Ric}^{\prime}(\widetilde{\nabla}^{\prime
}f)=\widetilde{Ric}^{\prime}(\xi^{\prime})\geq-k\text{.}\right.  $ Finally
using the $\widetilde{g}^{\prime}$-completeness and Theorem~\ref{tak 1} we
have $\lvert\widetilde{\Delta}^{\prime}f\lvert\leq k$ on each connected
component of $M^{\prime}$, so on $M^{\prime}$ itself. But $H^{\prime
}=-\widetilde{\operatorname{div}}^{\prime}(\widetilde{\nabla}^{\prime
}f)=\varepsilon\widetilde{\Delta}^{\prime}f$ ($\varepsilon=\pm1$ according to
the sign convention of the Laplacian), then $\lvert H^{\prime}\lvert
_{\widetilde{g}^{\prime}}\leq k.$ Using again that $\pi\ $is a local isometry
and $H^{\prime}=\pi^{\ast}H$ we get $\left\vert H\right\vert \leq k$.
\end{proof}

\begin{coro}
Let $(\overline{M}^{n+2},\overline{g})$ be a Lorentzian manifold and
$(M,\zeta)$ be a closed normalization of a totally umbilic null hypersurface
with umbilicity factor $\rho$ (i.e. $B=\rho g$) such that $\tau(\xi)=0$. If
there exists a nonnegative constant $k$ such that $\overline{Ric}(\xi)\geq-k$
and $M$ is $\widetilde{g}$-complete then it holds $\lvert\rho\lvert\leq
\frac{k}{n}$.
\end{coro}

\begin{proof}
The proof is straightforward using $H=n\rho$.
\end{proof}

Since any Riemannian metric on a compact manifold is complete, the following
also holds:

\begin{coro}
\label{thm1a} Let $(\overline{M},\overline{g})$ be a Lorentzian manifold
admitting $(M,\zeta)$ a closed normalization of a compact null hypersurface
and suppose that $\tau(\xi)=0$. Assume there exists a nonnegative constant $k$
such that $\overline{Ric}(\xi)\geq-k$. Then we have $\lvert H\lvert\leq k$
where $H$ stands for the (non normalized) mean curvature of $M$.
\end{coro}

\begin{thm}
\label{thm3} Let $(\overline{M}^{n+2},\overline{g})$ be a simply connected
Lorentzian manifold and $(M,\zeta)$ a closed normalization of a non compact
null hypersurface such that $\tau(\xi)=0$. Suppose $k$ is a positive constant
such that $\left.  \overline{Ric}(\xi)\geq-nk^{2}\right.  $. If $M$ is
$\widetilde{g}-$complete and $\lvert H\lvert=nk$, then the hypersurface $M$
endowed with the Riemannian structure $\widetilde{g}$ is isometric to the
warped product $\mathbb{R}\times_{e^{\pm kt}}Z$ where $Z$ inherits a
Riemannian structure from $\overline{M}$. In particular $M$ is totally umbilic.
\end{thm}

The proof makes use of the following.

\begin{thm}
[\cite{bi1}, Theorem~$1\cdot1$]\label{tak2} Let $(M,g)$ be a complete
connected Riemannian manifold such that there exists $f:M\longrightarrow
\mathbb{R}$ satisfying $\left\vert \nabla f\right\vert =1$. Suppose
\[
Ric(\nabla f,\nabla f)\geq-n\frac{{\phi}^{\prime\prime}(f(x))}{\phi(f(x))}%
\]
(resp $Ric(\nabla f,\nabla f)\geq-n\frac{({\phi^{\star}})^{\prime\prime
}(f(x))}{\phi^{\star}(f(x))})$.

If $\Delta f=-n\frac{{\phi}^{\prime}(f(x))}{\phi(f(x))}$ (resp $\Delta
f=n\frac{{\phi}^{\prime}(-f(x))}{\phi(-f(x))})$ then
\begin{align}
\Phi &  :\mathbb{R}\times_{\phi}Z\longrightarrow M\nonumber\\
&  (s,p)\longmapsto\psi_{s}(p)\nonumber
\end{align}
is an isometry (resp
\begin{align}
\Phi &  :\mathbb{R}\times_{\phi^{\star}}Z\longrightarrow M\nonumber\\
&  (s,p)\longmapsto\psi_{s}(p)\nonumber
\end{align}
is an isometry) where $\psi_{s}$ is the flow of $\nabla f$, $\phi
:\mathbb{R}\longrightarrow\mathbb{R}^{+}$ a smooth positive function,
$\phi^{\star}(t)=\phi(-t)$ and $Z=f^{-1}(0)$.
\end{thm}

\begin{proof}
(of Theorem~\ref{thm3})

We know that there exists $f:\overline{M}\longrightarrow\mathbb{R}$ such that
$\widetilde{\nabla}f=\xi$ and then $\widetilde{g}(\widetilde{\nabla
}f,\widetilde{\nabla}f)=1$. Let $\phi(t)=e^{kt}$ ($t\in\mathbb{R}$). It
follows that $\frac{{\phi}^{\prime\prime}(f(x))}{\phi(f(x))}=k^{2}$ and
$\frac{{\phi}^{\prime}(f(x))}{\phi(f(x))}=k$. We deduce using the assumption
on the ricci curvature that $\overline{Ric}(\xi)=\widetilde{Ric}(\xi
)\geq-n\frac{{\phi}^{\prime\prime}(f(x))}{\phi(f(x))}$ that is $\widetilde
{Ric}(\widetilde{\nabla}f,\widetilde{\nabla}f)\geq-n\frac{{\phi}^{\prime
\prime}(f(x))}{\phi(f(x))}$. Moreover, if $H=-nk$ then $H=\widetilde{\Delta
}f=-n\frac{{\phi}^{\prime}(f(x))}{\phi(f(x))}$. Using Theorem~\ref{tak2} we
conclude that $(M,\widetilde{g})$ is isometric to the warped product
$\mathbb{R}\times_{e^{kt}}Z$ where $Z=f^{-1}(0)$ is endowed with the induced
Riemannian metric. If $H=nk$, $(M,\widetilde{g})$ is isometric to the warped
product $\mathbb{R}\times_{e^{-kt}}Z$. Finally, since each $Z$ with the
Riemannian structure induced from the warped structure is totally umbilic in
$(M,\widetilde{g})$, the null hypersurface $M$ is totally umbilic.
\cite[Corollary 3.14]{bi0}.
\end{proof}

\begin{thm}
\label{iso} Let $(\overline{M},\overline{g})$ be a simply connected Lorentzian
manifold and $(M,\zeta)$ a $\widetilde{g}$-complete closed normalization of a
null hypersurface $M$, all of whose screen principal curvatures are constant.
Then it holds $\lvert H\lvert\leq\left\vert \overset{\ast}{A_{\xi}
}\right\vert ^{2}.$
\end{thm}

\begin{proof}
We know that $\overline{Ric}(\xi)=\xi(H)+\tau(\xi)H-\left\vert \overset{\ast
}{A_{\xi}}\right\vert ^{2}$ (\cite{bi11}). Observe that since the screen
principal curvatures are constant, $H$ and $\left\vert \overset{\ast}{A_{\xi}%
}\right\vert $ are constant quantities and that for closed normalizations,
$\widetilde{Ric}(\xi)=\overline{Ric}(\xi)-\tau(\xi)H$. Then,%
\[
\widetilde{Ric}(\xi)=-\left\vert \overset{\ast}{A_{\xi}}\right\vert
^{2}=constant.
\]

Using the fact that $\overline{M}$ is simply connected and $\zeta$ is closed,
we know there exists $f:M\longrightarrow\mathbb{R}$ such that $\widetilde
{\nabla}f=\xi$ and thus $\widetilde{g}(\widetilde{\nabla}f,\widetilde{\nabla
}f)=1$. By Theorem~\ref{tak 1} it follows that $\lvert\widetilde{\Delta
}f\lvert=\lvert H\lvert\leq\left\vert \overset{\ast}{A_{\xi}}\right\vert ^{2}$.
\end{proof}

\begin{coro}
\label{coriso1} Let $(\overline{M}^{n+2},\overline{g})$ be a simply connected
Lorentzian manifold and $(M,\zeta)$ a $\widetilde{g}$-complete closed
normalization of a non-totally geodesic null hypersurface $M$ all of whose
screen principal curvatures are non negative constants. Then at least one of
them is greater or equal to 1.
\end{coro}

\begin{proof}
Since all the eigenvalues of $\overset{\star}{A_{\xi}}$ are non-negative, the
inequality in Theorem~\ref{iso} becomes $H\leq\left\vert \overset{\ast}%
{A_{\xi}}\right\vert ^{2}.$ Let denote the eigenvalues by $\lambda_{i}$. If we
suppose that all of them are less than 1, then we have $\lambda_{i}\geq
\lambda_{i}^{2}$. But as $M$ is non-totally geodesic, there exist $i_{0}$ such
that $\lambda_{i_{0}}$ is positive and then $\lambda_{i_{0}}>\lambda_{i_{0}
}^{2}$. It follows that $H>\left\vert \overset{\ast}{A_{\xi}}\right\vert ^{2}
$, which is in contradiction with $H\leq\left\vert \overset{\ast}{A_{\xi}
}\right\vert ^{2}.$
\end{proof}

\begin{coro}
\label{coriso2} Let $(\overline{M}^{n+2},\overline{g})$ be a simply connected
Lorentzian manifold and $(M,\zeta)$ a $\widetilde{g}$-complete closed
normalization of a proper totally umbilic null hypersurface $M$ with constant
umbilicity factor $\rho$. Then $|\rho|\geq1$.
\end{coro}

\begin{proof}
Because $\rho$ is constant, all the eigenvalues of $\overset{\ast}{A_{\xi}}$
are constant. Using theorem (\ref{iso}), we have $\left\vert H\right\vert
\leq\left\vert \overset{\ast}{A_{\xi}}\right\vert ^{2}.$ But since $M$ is
proper totally umbilic, $H=n\rho\not =0$ and $\left\vert \overset{\ast}%
{A_{\xi}}\right\vert ^{2}=n\rho^{2}$ so the inequality becomes $n\left\vert
\rho\right\vert \leq n\rho^{2}$ and we get $\left\vert \rho\right\vert \geq1$.
\end{proof}

\subsection{Null hypersurfaces with semi-definite shape $B$}

\label{semidefinitesharp} Positive definiteness of the second fundamental form
of hypersurfaces have many consequences in Riemannian geometry. A well-known
theorem due to Hadamard ~\cite{hadamard} (see also ~\cite[Theorem~2.4]%
{dajczer}) states that if the second fundamental form of a compact immersed
hypersurface $M$ of a Euclidean space is positive definite, then $M$ is
embedded as the boundary of a convex body. This also implies an equivalence
between definiteness of the second fundamental form and the fact that $M$ is
orientable and its (spherical) Gauss map is a diffeomorphism. Equivalently,
the Gaussian curvature of $M$ is nowhere-vanishing. In the present section, we
present some facts about definiteness of the second fundamental form $B$ and
geodesics relative to the rigged Riemannian structure. Recall that if a
complete Riemannian manifold supports a convex function, the latter is
constant along any closed geodesic, \cite[Proposition 2.1]{bi6}.

\begin{prop}
\label{prop3.1} Let $(M,\zeta)$ be a closed normalization of a null
hypersurface $M$ in a simply connected Lorentzian manifold $(\overline
{M},\overline{g})$. Assume $M$ is $\widetilde{g} $-complete and $B$ restricts
to a definite form on $\mathscr{S}(\zeta)$. Then $(M,\widetilde{g})$ contains
no closed geodesics.
\end{prop}

\begin{proof}
By a change of rigging $\zeta\longleftarrow-\zeta$ if neccessary, we can
suppose without loss of generality that $M$ is connected and that the
restriction of $B$ to $\mathscr{S}(\zeta)$ is negative definite which implies
that $B$ is negative semi-definite on$M$. Also, by the simply connectedness of
$\overline{M}$ and the closedness of the normalization, we know that there
exists $f:M\longrightarrow\mathbb{R}$ such that $\widetilde{\nabla}f=\xi$. Let
us remark that fibers of $f$ are leaves of $\mathscr{S}(\zeta)$ (see proof of
Theorem\ref{main}). Using \cite[Proposition 3.15.]{bi0}, we have
$\widetilde{Hess}f(U,V)=-B(U,V)$ for all $U,V\in TM$. $B$ being negative
semi-definite, $f$ is a convex function on $M$. Suppose that there exists a
closed geodesic $\gamma$ in $M$. Then $f\circ\gamma$ is a constant say $c$ (as
stated above \cite{bi6}) and $\gamma$ is contained in the leaf $f^{-1}(c)$ of
$\mathscr{S}(\zeta)$, hence $\gamma\prime\in\mathscr{S}(\zeta)$ and then
$\widetilde{Hess}f((\gamma\prime, \gamma\prime)> 0$, which gives the
contradiction as $\widetilde{Hess}f(\gamma\prime, \gamma\prime) =
(f\circ\gamma)\prime\prime= 0$; $f\circ\gamma$ being constant. We conclude
that $(M, \widetilde{g})$ contains no closed geodesics.
\end{proof}

\begin{rem}
\label{remclosedg} The proposition remain true if $\overline{M}$ is not simply
connected but the first De Rham cohomology group $H^{1}(M, \mathbb{R})$ is
trivial or the one form $\omega$ is exact so that the rigged vector field
$\xi$ is a gradient vector field.
\end{rem}

Since for proper totally umbilic null hypersurfaces the restriction of $B$ to
the screen structure $\mathscr{S}(\zeta)$ is always definite form, we easily
deduce the following.

\begin{coro}
\label{closedgeod} Let $(M,\zeta)$ be a closed normalization of a\ proper
totally umbilic null hypersurface $M$ in a simply connected Lorentzian
manifold $(\overline{M},\overline{g})$ such that $M$ is $\widetilde{g}%
$-complete. Then $(M,\widetilde{g})$ contains no closed geodesics.
\end{coro}

The next result gives a restriction on the topology of proper totally umbilic
null surface (in 3-dimensional Lorentzian manifold) which can admit a $\left.
\widetilde{g}\text{-complete}\right.  $ Riemannian metric for a given rigging.

\begin{thm}
\label{thm4b} Let $(M,\zeta)$ be a closed normalization of a null surface $M$
non totally geodesic at any point in a simply connected 3-dimensional
Lorentzian manifold $(\overline{M}^{3},\overline{g})$ such that $M$ is
$\widetilde{g}$-complete, then $M$ is homeomorphic to the plane or the cylinder.
\end{thm}

\begin{proof}
From corollary~\ref{closedgeod} the null hypersurface $(M, \widetilde{g})$
contains no closed geodesics. It follows from the classification of complete
surfaces without closed geodesic (see \cite{GT}, Theorem~3.2) that $M$ is
homeomorphic to the plane or the cylinder.
\end{proof}

\begin{coro}
\label{classi} Let $\overline{M}=\mathbb{R}\times_{f}L$ be a 3-dimensional
Generalized-Robertson-Walker (GRW) space with complete Riemannian factor
$(L,g_{0})$. Any topologically closed null surface and non totally geodesic at
any point is homeomorphic to the plane or the cylinder.
\end{coro}

\begin{proof}
Let $M$ be a topologically closed proper totally umbilic null surface.
Consider the normalizing rigging $\zeta= \sqrt{2}\frac{\partial}{\partial t}$
for $M$. Then from Theorem~\ref{grw3}, $(M, \widetilde{g})$ is complete.
Moreover the $1-$form $\omega$ is exact. From Remark~\ref{remclosedg} and
Theorem~\ref{thm4b}, $M$ is homeomorphic to the plane or the cylinder.
\end{proof}

\subsection{Totally geodesic null hypersurfaces in Robertson-Walker spaces}

Totally geodesic null hypersufaces are intensively used in general relativity
as they represent horizons of various sorts (Non expanding horizon, isolated
horizon, Killing horizon, etc.). We investigate here the existence of totally
geodesic null hypersurfaces in Robertson-Walker spaces. For this, we use the
Hilbert theorem which we recall here.

\begin{thm}
(\cite{Hilb}, \cite{JG})\label{Hilbert} Let $\Sigma$ be a complete surface
with negative constant curvature $K$. Then, there exists no isometric
immersion $\left.  f:\Sigma\rightarrow\mathbb{M}^{3}(c)\right.  $ (with $K<-1$
for $c=-1$), where $\mathbb{M}^{3}(c)$ stands for the simply-connected
complete Riemannian $\left.  \text{3-space}\right.  $ with constant sectional
curvature $c=-1,0,1$.
\end{thm}

We prove the following.

\begin{thm}
\label{rober} Let $\overline{M}=\mathbb{R}\times_{f}\mathbb{M}^{3}(c)$ be a
Robertson-Walker space. Then the followings hold:

\begin{itemize}
\item[1.] If $c=0 $ or $c=-1$ and $f$ is strictly monotone, then there is no
topologically closed totally geodesic null hypersurface in $\overline{M}$.

\item[2.] If $c=1$ and $f^{\prime}(t)>1\, \forall\; t$, then there is no
topologically closed totally geodesic null hypersurface in $\overline{M}$.
\end{itemize}
\end{thm}

\begin{proof}
Let $\overline{M}=\mathbb{R}\times_{f}\mathbb{M}^{3}(c)$ be a Robertson-Walker
space. Suppose there exist a topologically closed totally geodesic null
hypersuface $M$ in $\overline{M}$. Consider the normalizing rigging
$\zeta=\sqrt{2}\frac{\partial}{\partial t}$ for $M$. Let, $\Pi:\mathbb{R}%
\times\mathbb{M}^{3}(c)\rightarrow\mathbb{R}$ be the projection on the first
factor. Then a leaf of $\mathscr{S}(\zeta)$ is the intersection of $M$ with a
fiber of $\Pi$, hence a leaf of $\mathscr{S}(\zeta)$ is a closed subset
contained in some slice $\{t_{0}\}\times\mathbb{M}^{3}(c)$. Let us call $g$
the Riemannian metric on $\mathbb{M}^{3}(c)$. Since the Lorentzian metric on
$\overline{M}$ is given by $\overline{g}=-dt^{2}+f^{2}g$, using the equalities
in (\ref{garc}) we get that $\widetilde{g}_{\lvert_{S}}=f^{2}(t_{0})g$ so that
$(S,\widehat{g}=\frac{1}{f^{2}(t_{0})}\widetilde{g}_{\lvert_{S}})$ is a
complete surface isometrically immersed in $\mathbb{M}^{3}(c)$. Elsewhere, the
following relation (\cite{bi0}) holds:
\begin{align*}
&  \overline{K}(X,Y)\\
&  =\widetilde{K}^{S}(X,Y)-C(X,X)B(Y,Y)-B(X,X)C(Y,Y)+2C(X,Y)B(X,Y)
\end{align*}
$\forall\;X,Y\in\mathscr{S}(\zeta)$, where $\overline{K}(X,Y)$ is the
sectional curvature in $(\overline{M},\overline{g})$ and $\widetilde{K}%
^{S}(X,Y)$ the induced sectional curvature from $(M,\widetilde{g})$ that is
$\widetilde{{K}}^{S}(X,Y)$ is the Gauss curvature of $(S,\widetilde{g}%
_{\lvert_{S}})$. Since the null hypersurface $M$ is totally geodesic, $B=0$.
Hence $\widetilde{K}^{S}(X,Y)=\overline{K}(X,Y)=\frac{c-(f^{\prime})^{2}%
}{f^{2}}$. Note that $S$ being contained in some slice $\{t_{0}\}\times
\mathbb{M}^{3}(c)$, $\widetilde{K}^{S}(X,Y)=\frac{c-(f^{\prime})^{2}(t_{0}%
)}{f^{2}(t_{0})}$ and is constant. Finally since $\widehat{g}=\frac{1}%
{f^{2}(t_{0})}\widetilde{g}_{\lvert_{S}}$, the Gauss curvature of the surface
$(S,\widehat{g})$ is $\widehat{K}=c-(f^{\prime})^{2}(t_{0})$ If $c=0$ or
$c=-1$ and $f$ is strictly monotone or $c=1$ and $f^{\prime}(t)>1\,\forall\;t$
then $(S,\widehat{g})$ has negative constant Gauss curvature (with
$\widehat{K}<-1$ in case $c=-1$). The contradiction follows from the Hilbert Theorem.
\end{proof}

In (\cite{Ga}, TheoremIV.1), Galloway shows that if a Lorentzian manifold is
null complete and satisfy the null convergence condition then any null line is
conained in a smooth (topologically) closed achronal totally geodesic null
hypersurface. So under null completeness and null convergence condition
hypothesis, the absence of topologically closed totally geodesic null
hypersurface implies the absence of null line. The following holds:

\begin{coro}
\label{corline} Let $\overline{M}=\mathbb{R}\times_{f}\mathbb{M}^{3}(c)$ be a
null complete Robertson-Walker space satifying the null convergence condition.
Then we have:

\begin{itemize}
\item[1.] If $c=0$ or $c=-1$ and $f$ is strictly monotone, $\overline{M}$
contains no lightlike line.

\item[2.] If $c=1$ and $f^{\prime}(t)>1\,\forall\;t$, $\overline{M}$ contains
no lightlike line.
\end{itemize}
\end{coro}

\textbf{Aknowledgment.} The second and third authors has been partially
supported by a FEDER-MTM2016-78647-P grant. The third author has been
partially supported by a CEA-SMA grant.



\begin{thebibliography}{99}                                                                                               %


\bibitem {atbb1}C. Atindogbe and L. Berard-Bergery. Distinguished
normalization on non-minimal null hypersurfaces. \textit{Math. Sci. Appl.
E-Notes,} \textbf{1} (2013) 18-35.


\bibitem {At}{C. Atindogbe}, Blaschke type normalization on lightLike
Hypersurfaces, \textit{Journal of Mathematical Physics, Analysis, Geometry,}
\textbf{6} (2010) 362-382.

\bibitem {ATE}C. Atindogbe, J. -P. Ezin and J. Tossa, Pseudo-inversion of
degenerate metrics. \textit{Int. J. Math. Math. Sci.}, \textbf{55} (2003) 3479-3501.

\bibitem {ad1}C. Atindogbe and K. L. Duggal, Conformal screen on lightlike
hypersurfaces. \textit{Int. J. Pure Appl. Math}. \textbf{11} (2004) 421-441.

\bibitem {bi11}C. Atindogbe and H. T. Fosting, Newton transformations on null
hypersurfaces. \textit{Comm. Math.} \textbf{23} (2015) 57-83.

\bibitem {bi6}R. L. Bishop and B. O'neill, Manifolds of negative curvature.
\textit{Trans. Amer. Math. Soc.} \textbf{145} (1969) 1-49.

\bibitem {DL}{D. L. Johnson and L. B. Whitt }, Totally geodesic foliations,
\textit{J. Differential Geom.} \textbf{15} (1980) 225-235.

\bibitem {Ga}{G.J. Galloway },{ Maximum principles for null hypersurfaces and
null splitting theorems}. \textit{Ann. Henri Poincar\'{e}} \textbf{1} (2000) 543-567

\bibitem {JG}{ Jos\'{e} A. G\'{a}lvez}, Surfaces of constant curvature in
3-dimensional space forms, \textit{Mat. Contemp}. \textbf{37} (2009) 1-42.

\bibitem {Hilb}{ D. Hilbert}, \"{U}ber Flachen von constanter Gausscher
Krummung, (German) (On surfaces of constant Gaussian curvature),
\textit{Trans. Amer. Math. Soc}. \textbf{2} (1901) 87-99.

\bibitem {DB}K. L. Duggal and A. Bejancu. \textit{Lightlike Submanifolds of
Semi-Riemannian Manifolds and Applications}, Kluwer Academic, 1996.



\bibitem {bi4}M. Fern\'{a}ndez-L\'{o}pez, E. Garcia-Rio, D. N. Kupeli and B.
\"{U}nal. A curvature condition for a twisted product to be a warped product.
\textit{Manuscripta Math.} \textbf{106} (2001) 213-217.

\bibitem {bi12}A. E. Fisher, Riemannian submersions and the regular interval
theorem of Morse theory. \textit{Ann. Global Anal. Geom}. \textbf{14} (1996) 263--300.

\bibitem {bi5}W B.Gordon, An analytical criterion for the completness of
Riemannian manifolds. \textit{Proc. Amer. Math. Soc.} \textbf{37} (1973) 221-225.

\bibitem {bi5 b}W. B. Gordon, Corrections to "An Analytical Criterion for the
Completeness of Riemannian Manifolds". \textit{Proc. Amer. Math. Soc.}
\textbf{45} (1974) 130-131.

\bibitem {dajczer}M. Dajczer, M. Antonucci, G. Oliveira, P. Lima-Filho, R.
Tojeiro, \textit{Submanifolds and isometric immersions}, Math. Lectures
Ser.~13, Publish or Perish Inc. 1990.

\bibitem {hadamard}Hadamard, J., Les surfaces \`{a} courbures oppos\'{e}es et
leurs lignes g\'{e}od\'{e}siques, J. Math. Pures Appl. \textbf{4} (1896) 27-73.

\bibitem {GutOlea2003a}M. Guti\'{e}rrez and B. Olea, Global decomposition of a
manifold as a generalized Robertson-Walker space, \textit{Differential Geom.
Appl}. \textbf{27} (2009) 146--156.

\bibitem {GutOlea2003b}M. Guti\'{e}rrez and B. Olea, Splitting theorems in
presence of an irrotational vector field. \textit{arXiv:math/0306343}, 2003.

\bibitem {bi0}M.{\ Guti\'{e}rrez and B. Olea}. {Induced Riemannian structures
on null hypersurfaces.} \textit{Math. Nachr.} \textbf{289} (2016) 1219-1236.

\bibitem {GT}G. Thorbegsson, Closed geodesic on non compact Riemannian
manifolds, \textit{Math. Z.} \textbf{159} (1978) 249-258.

\bibitem {kupeli}{D. N. Kupeli,} Degenerate submanifolds in semi-Riemannian
geometry. \textit{Geom. Dedicata}, \textbf{24} (1987) 337-361.

\bibitem {bi1}T. {Sakai.} Warped products and Riemannian manifolds admitting a
function whose gradient is of constant norm. \textit{Math. J. Okayama Univ.}
\textbf{39} (1997) 165-185.

\bibitem {bi2}T. Sakai. On Riemannian manifolds admitting a function whose
gradient is of constant norm. \textit{Kodai Math. J.} \textbf{19} (1996) 39-51.
\end{thebibliography}
\end{document}